\documentclass[12pt,reqno,oneside]{amsart}
\usepackage{amsmath,amsthm,amsfonts,amssymb}
\usepackage{indentfirst}
\usepackage{url}
\usepackage{hyperref}
\usepackage{graphicx}
\usepackage{color}
\usepackage{soul} 

\theoremstyle{plain}
\newtheorem{theorem}{Theorem}[section]
\newtheorem{cor}[theorem]{Corollary}
\newtheorem{lem}[theorem]{Lemma}
\newtheorem{prop}[theorem]{Proposition}

\theoremstyle{definition}
\newtheorem{defn}[theorem]{Definition}

\newtheorem{obs}[theorem]{Remark}

\numberwithin{equation}{section}
\numberwithin{figure}{section}

\begin{document}

\baselineskip=18pt

\title[A contact process with stronger mutations]{A contact process with stronger mutations on trees}

\author[Fabio Lopes]{Fabio Lopes}
\address[F. Lopes]{Departamento de Matem\'atica, Universidad Tecnol\'ogica Metropolitana, Chile}
\email{f.marcellus@utem.cl}

\author[Alejandro Rold\'an]{Alejandro Rold\'an-Correa}
\address[A. Rold\'an]{Instituto de Matematicas, Universidad de Antioquia, Colombia}
\email{alejandro.roldan@udea.edu.co}

\thanks{Research supported by ANID-FONDECYT Iniciación grant (11230220) and Universidad de Antioquia.}

\date{\today}

\begin{abstract} 
We consider a spatial stochastic model for a pathogen population growing inside a host that attempts to eliminate the pathogens through its immune system. The pathogen population is divided into different types. A pathogen can either reproduce by generating a pathogen of its own type or produce a pathogen of a new type that does not yet exist in the population. Pathogens with living ancestral types are protected against the host's immune system as long as their progenitors are still alive. Each pathogen type without living ancestral types is eliminated by the immune system after a random period, independently of the other types. When a pathogen type is eliminated from the system, all pathogens of this type die simultaneously. In this paper, we determine the conditions on the set of model parameters that dictate the survival or extinction of the pathogen population when the dynamics unfold on graphs with an infinite tree structure.

\end{abstract}

\maketitle

\section{Introduction} \label{S: Introduction}
Pathogens are microorganisms that cause disease by invading a host organism, where they find favorable conditions for replication. The host's immune system acts as a defense mechanism, identifying and eliminating these invaders. However, pathogens can acquire mutations that enable them to adapt to new conditions and evade immune detection, complicating their elimination and potentially leading to persistent infections.  Several stochastic models have been proposed to understand persistence and extinction in biological populations under disturbances, mutation, and immune response. Notably, Bertacchi, Zucca and Ambrosini \cite{BZA2016} and Zucca \cite{FZ2014} have examined how populations adapt their timing of life-history events under environmental disturbances and how bacterial persistence arises under antibiotic treatments. Cox and Schinazi \cite{Cox} and Schinazi \cite{schinazi2} have focused on the role of mutation in viral survival, showing that populations of ever-changing mutants may persist even beyond classical extinction thresholds. To investigate whether a pathogen can evade the immune system solely due to a high mutation rate, Schinazi and Schweinsberg~\cite{SS2008} introduced three mathematical models. These models assume that pathogens mutate to produce new variants and that the immune system eradicates all pathogens of a given type simultaneously after a random period. The main difference among the three models lies in how the immune system operates. For a broader overview of stochastic models in biology, see \cite{LanchierBook}. 

Among the models proposed by Schinazi and Schweinsberg~\cite{SS2008}, Models 2 and S2 are particularly noteworthy, as they extend Harris’s Contact Process~\cite{H1974}. Model 2 is a non-spatial version, while Model S2 represents a spatial version on \(\mathbb{Z}^d\). These models describe a pathogen population that evolves by generating new pathogens, which may either be of the same type as their ancestral pathogen or a completely new type (mutation), distinct from all existing ones in the population. In this framework, whenever a new pathogen type appears for the first time, it survives for a random duration, independently of other types, before all pathogens of that type are simultaneously eliminated. Schinazi and Schweinsberg~\cite{SS2008} established conditions for phase transitions (survival or extinction) in Models 2 and S2. Later, Liggett et al.~\cite{LSS2008} studied the spatial version of Model S2 on trees, determining conditions not only for phase transition but also for weak survival.  

Recently, Grejo et al.~\cite{GLMR2024} introduced a variant of Models 2 and S2 by incorporating evolutionary dynamics into mutations. This variant assumes that each new pathogen type (mutation) is stronger than its ancestral type (evolution), requiring the immune system to eliminate the ancestral type before it can target the new one. Specifically, the model in Grejo et al.~\cite{GLMR2024} describes a pathogen population that evolves by generating new pathogens, which may either be of the same type as their ancestor or, due to mutations, a stronger type distinct from all others in the population. All pathogens of a given type are simultaneously eliminated by the host’s immune system after a random time, provided their ancestral pathogens are no longer present. Meanwhile, pathogens with living ancestors remain protected until their progenitors are eliminated. In this framework, mutations are considered beneficial (i.e., making pathogens stronger).  

In the non-spatial model of Grejo et al.~\cite{GLMR2024}, pathogens reproduce independently at rate \(\lambda > 0\). Upon birth, a new pathogen inherits its parent's type with probability \(1 - r\) or undergoes a beneficial mutation with probability \(r\), producing a stronger type. The immune system responds independently at rate 1, eradicating all pathogens of a given type once no ancestral pathogens of that type remain. The spatial version of this model follows similar birth, mutation, and death dynamics but evolves on a graph, where each pathogen occupies a site and can only place offspring in adjacent empty sites.  
Grejo et al.~\cite{GLMR2024} established that in the non-spatial model, pathogens survive with positive probability if and only if \(\lambda > (1 + \sqrt{r})^{-2}\). For the spatial model on \(\mathbb{Z}^d\), the behavior differs from the non-spatial case: if \(\lambda\) is small, pathogens die out with probability 1, but if \(\lambda\) is sufficiently large, two scenarios emerge, pathogens survive with positive probability for large \(r\) but die out with probability 1 for small \(r\).  

In this paper, we study the spatial version of the model proposed in \cite{GLMR2024} on graphs with an infinite tree structure. This work is organized as follows. In the next section, we formally define the spatial model on general graphs and establish phase transition results for survival probability when the graph is an infinite tree. Finally, in Section~\ref{proofs}, we provide proofs for the results presented in Section~\ref{modelandresult}.

\section{Spatial model}\label{modelandresult}

We consider the evolution of a population of pathogens occurring on a graph $\mathcal{G}$. The dynamics of the model are as follows. Each vertex of $\mathcal{G}$ can either be occupied by a pathogen or be empty. We assume that at time $t=0$, there is a single pathogen of type 1 on a vertex of $\mathcal{G}$, which we call the  root of $\mathcal{G}$. For a vertex \(x\) occupied by a pathogen and \(y\) one of its nearest neighbors (out-neighbors if $\mathcal{G}$ is a directed graph), the pathogen on \(x\), after an exponential time with rate \(\lambda\), gives birth to a pathogen on \(y\), provided \(y\) is empty. If \(y\) is occupied, nothing happens. This new pathogen will be of the same type as the pathogen on \(x\) with probability \(1-r\). On the other hand, with probability \(r\), a mutation will occur and the new pathogen will be of a different type, one that has not appeared so far. We consider the pathogen present at time zero to be type 1, and the \(k\)-th type to appear will be called type \(k\). To each new type, we associate an independent exponential clock with rate 1, which will start ticking only when its progenitor dies. When the clock of a given type rings, all pathogens of that type are simultaneously eliminated  by the immune system. These clocks can be interpreted as the incremental time (or killing time) required for the immune system to recognize and eliminate a new pathogen type after having eliminated its progenitor. Once a type is recognized, the immune system is capable of eradicating all pathogens of that type.  We denote this model by $\{\mathcal{G},\lambda,r\}$.

Observe that the model \(\{\mathcal{G}, \lambda, r\}\) is a continuous-time stochastic process with state space \(\{0, 1, \ldots\}^{\mathcal{V(G)}}\), where \(\mathcal{V(G)}\) denotes the vertex set of \(\mathcal{G}\). The evolution of this process (the status at time \(t\)) is denoted by \(\eta_t\). Specifically, the status of a vertex \(x\) at time \(t\) can be either \(\eta_t(x) = 0\) (empty) or \(\eta_t(x) = k\) (occupied by a pathogen of type \(k\)).

\begin{obs}  
The model \(\{\mathcal{G},\lambda,r\}\) follows the same birth and mutation dynamics as model S2 from Schinazi and Schweinsberg~\cite{SS2008}. The key difference lies in the pathogen elimination mechanism. In model S2, when a new pathogen type appears in the population, it is also assigned an independent killing time - an exponential clock with rate 1. When this clock rings, all pathogens of that type are simultaneously eliminated by the immune system. Unlike in \(\{\mathcal{G},\lambda,r\}\), where the killing times start ticking only when the progenitor type dies, here they begin ticking from the moment the new pathogen type is created.  \\
Moreover, in \(\{\mathcal{G},\lambda,r\}\), the lifetimes of pathogen types are not all independent of each other, as they are in model S2. Consequently, \(\{\mathcal{G},\lambda,r\}\) is not a Markov process, whereas model S2 does satisfy the Markov property.  
\end{obs}  

\begin{defn} Consider the process $\{\mathcal{G},\lambda,r\}$. 
If all pathogens are eventually removed from $\mathcal{G}$ with probability 1, we say that the process {\it dies out}. Otherwise, we say that the process {\it survives}.
\end{defn}

\begin{obs}  
A high birth rate combined with a low mutation rate leads to many vertices being occupied by pathogens of the same type. As a result, birth mutations become less likely since most vertices are already occupied, making it harder for the pathogen population to survive. Consequently, the survival probability of \(\{\mathcal{G},\lambda,r\}\) is not necessarily a non-decreasing function of \(\lambda\).  
Similarly, at low mutation rates, more vertices can become available when death events occur (i.e., when all pathogens of a given type are eliminated), allowing new pathogens to occupy the vacated spaces. This suggests that the survival probability of \(\{\mathcal{G},\lambda,r\}\) is not necessarily a non-decreasing function of \(r\) either. This behavior contrasts sharply with the non-spatial version of the model \cite[non-spatial model]{GLMR2024}, where the survival probability is a non-decreasing function of both \(\lambda\) and \(r\).  
\end{obs}

In this work, we study the model $\{\mathcal{G}, \lambda, r\}$ on two types of graphs. The first one, $\mathcal{G} = \mathbb{T}_d$ with $d \geq 1$, is an infinite homogeneous rooted tree in which each vertex has $d+1$ nearest neighbors. 
The second one, $\mathcal{G} = \mathbb{T}_d^+$, is an infinite directed rooted tree where each vertex has $d$ out-neighbors and one in-neighbor, except for the root, which has only $d$ out-neighbors. Naturally, the dynamics of the model $\{\mathbb{T}_d, \lambda, r\}$ are more complex than those of $\{\mathbb{T}_d^+, \lambda, r\}$. In $\{\mathbb{T}_d^+, \lambda, r\}$, a vertex can be occupied by a pathogen only once; once the pathogen dies, the vertex remains empty forever. Additionally, a pathogen born at a site \( x \) can only be a descendant of a pathogen located at a neighboring vertex closer to the root than \( x \). In contrast, in $\{\mathbb{T}_d, \lambda, r\}$, a vertex can be occupied multiple times by different pathogens, with their progenitors potentially located at any neighboring vertex. Due to the specific dynamics of the model on directed trees, the following monotonicity property is satisfied.

\begin{prop}\label{monotonia}
The survival probability in \(\{\mathbb{T}_d^+, \lambda, r\}\) is a non-decreasing function of both \(\lambda\) and \(r\).
\end{prop}

Proposition~\ref{monotonia} allows us to establish the following results on phase transitions for the survival probability in the model \(\{\mathbb{T}_d^+, \lambda, r\}\).

\begin{theorem}\label{T:arbol_dirigido}   
Consider the process \(\{\mathbb{T}_d^+, \lambda, r\}\) for $d\geq1.$ Define 
\begin{equation}\label{lambda_c}
  \lambda_c(d,r) := \left[\frac{\sqrt{d-1}-\sqrt{rd}}{d(1-r)-1}\right]^{2}.
\end{equation}

\begin{itemize}
    \item[$(i)$] If $\lambda<\lambda_c(d,r) $ then the process $\{\mathbb{T}_d^+,\lambda,r\}$ dies out.  
    \item[$(ii)$] If $\lambda>\lambda_c(d,r) $ then the process $\{\mathbb{T}_d^+,\lambda,r\}$ survives. 
\end{itemize}
\end{theorem}

Note that the non-spatial model defined in \cite{GLMR2024}, constrained such that each pathogen can generate at most $d$ new pathogens, corresponds to the model $\{\mathbb{T}_d^+,\lambda/d,r\}$. According to Theorem \ref{T:arbol_dirigido}, the critical parameter for this model is given by $\lambda'_c(d,r) = d\lambda_c(d,r)$, from which it follows that $\lambda'_c(d,r) \to (1+\sqrt{r})^{-2}$ as $d \to \infty$. This result aligns with \cite[Theorem 2.2]{GLMR2024}.\\

Using the Theorem~\ref{T:arbol_dirigido} and an appropriate coupling, we obtain as corollary our main result. 

\begin{cor}\label{C:arbol}  
Consider the process \(\{\mathbb{T}_d, \lambda, r\}\) for \(d \geq 1\), and \(\lambda_c(d, r)\)  defined as in (\ref{lambda_c}).  
\begin{itemize}
    \item[$(i)$] If $\lambda<\lambda_c(d+1,r)$ then the process $\{\mathbb{T}_d,\lambda,r\}$ dies out.
    \item[$(ii)$] If $\lambda>\lambda_c(d,r)$ then the process $\{\mathbb{T}_d,\lambda,r\}$ survives. 
\end{itemize}
\end{cor}

\begin{obs}
Liggett et al.~\cite{LSS2008} studied the spatial version of Model S2, introduced by Schinazi and Schweinsberg~\cite{SS2008}, on the tree \(\mathbb{T}_d\), with \( d \geq 2 \). They established that the process survives for all \( r \in (0,1] \) if \( \lambda > \frac{1}{d-1} \) and dies out if \( \lambda \leq \frac{1}{d-1+2r} \). 
On the other hand, according to Corollary \ref{C:arbol}, our model \(\{\mathbb{T}_d, \lambda, r\}\), with \( d \geq 2 \), survives when \( \lambda > \lambda_c(d) = \frac{1}{(\sqrt{d-1}+ \sqrt{rd})^2} \) and dies out if \( \lambda < \lambda_c(d+1) = \frac{1}{(\sqrt{d}+ \sqrt{r(d+1)})^2} \). Therefore, for \( d \geq 2 \), we observe that in the range \( \frac{1}{(\sqrt{d-1}+ \sqrt{rd})^2} < \lambda \leq \frac{1}{d-1+2r} \), Model S2 on \(\mathbb{T}_d\) dies out, whereas our model \(\{\mathbb{T}_d, \lambda, r\}\) survives. See Figure \ref{figura1} for a graphical representation. 
This contrast highlights an interesting phase transition, emphasizing how a mechanism involving stronger mutations can significantly influence large-scale behavior, where seemingly minor differences in the dynamics can lead to drastically different macroscopic outcomes.
\end{obs}

\begin{figure}[ht]
	\begin{tabular}{ccc}
		$\lambda$ & \parbox[c]{10cm}{\includegraphics[trim={0cm 0cm 0cm 0cm}, clip, width=10cm]{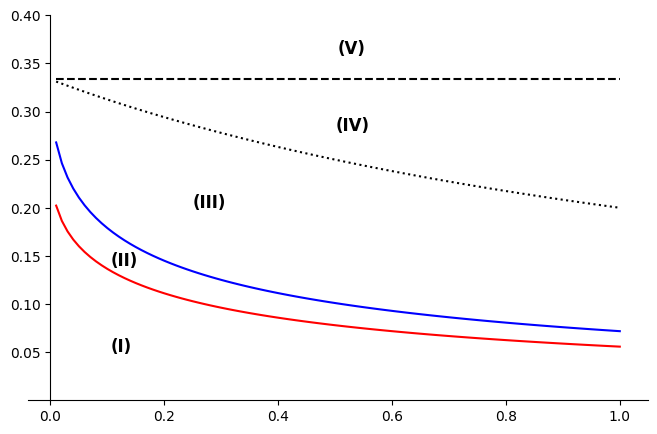}} & 
		\begin{tabular}{l}
			\textbf{\textcolor{red}{------}} $\lambda=\frac{1}{(\sqrt{d}+\sqrt{r(d+1)})^2}$\\ \\
			\textbf{\textcolor{blue}{------}} $\lambda=\frac{1}{(\sqrt{d-1}+\sqrt{rd})^2}$\\ \\
			\textbf{-\,-\,-\,-\,-} $\lambda=\frac{1}{d-1}$\\ \\
			\textbf{$\cdots \cdots$} $\lambda=\frac{1}{d-1+2r}$
		\end{tabular} \\
		& $r$
	\end{tabular}
	\caption{Model S2 on \(\mathbb{T}_4\) vs Model \(\{\mathbb{T}_4, \lambda, r\}\).   
In region (I), both models die out.  
In region (II), Model S2 dies out, while the behavior of Model \(\{\mathbb{T}_4, \lambda, r\}\) remains inconclusive.  
In region (III), Model S2 dies out, whereas Model \(\{\mathbb{T}_4, \lambda, r\}\) survives. In region (IV), Model \(\{\mathbb{T}_4, \lambda, r\}\) survives, while the behavior of Model S2 is inconclusive.  
In region (V), both models survive.  } 
	\label{figura1}
\end{figure}

In Figure \ref{figura2}, we illustrate the behavior of Model \(\{\mathbb{T}_4, \lambda, r\}\) based on Monte Carlo simulations.  These experiments allow us to examine more closely the dynamics of the model in region (II) of Figure \ref{figura1}, where analytical results were inconclusive. We considered values $r\in [0.01,1]$ (in steps of $0.02$) and $\lambda\in [0.01,0.4]$ (in steps of $0.01$), and for each parameter pair $(r,\lambda)$ we performed 30000 independent realizations of the process. 

To ensure comparability between parameter regimes and to avoid unbounded growth in supercritical cases, each simulation was truncated when either the total propagation time reached 30 or the number of explored nodes exceeded 3500. These thresholds do not affect extinction outcomes but provide a practical stopping criterion once survival becomes highly probable. Each point in Figure \ref{figura2} represents the empirical outcome for a specific parameter pair 
$(r,\lambda)$. Pink circles indicate total extinction in all 30000 realizations. The exclusive presence of pink circles in Region~(II) indicates that the process becomes extinct or has an extremely small probability of survival throughout this region. Nevertheless, a more exhaustive simulation study would be required to draw definitive conclusions about the survival–extinction behavior within Region~(II). All the computations were conducted in the R software \cite{R}.

\begin{figure}[ht]
	\includegraphics[trim={0cm 0cm 0cm 0cm}, clip, width=16cm]{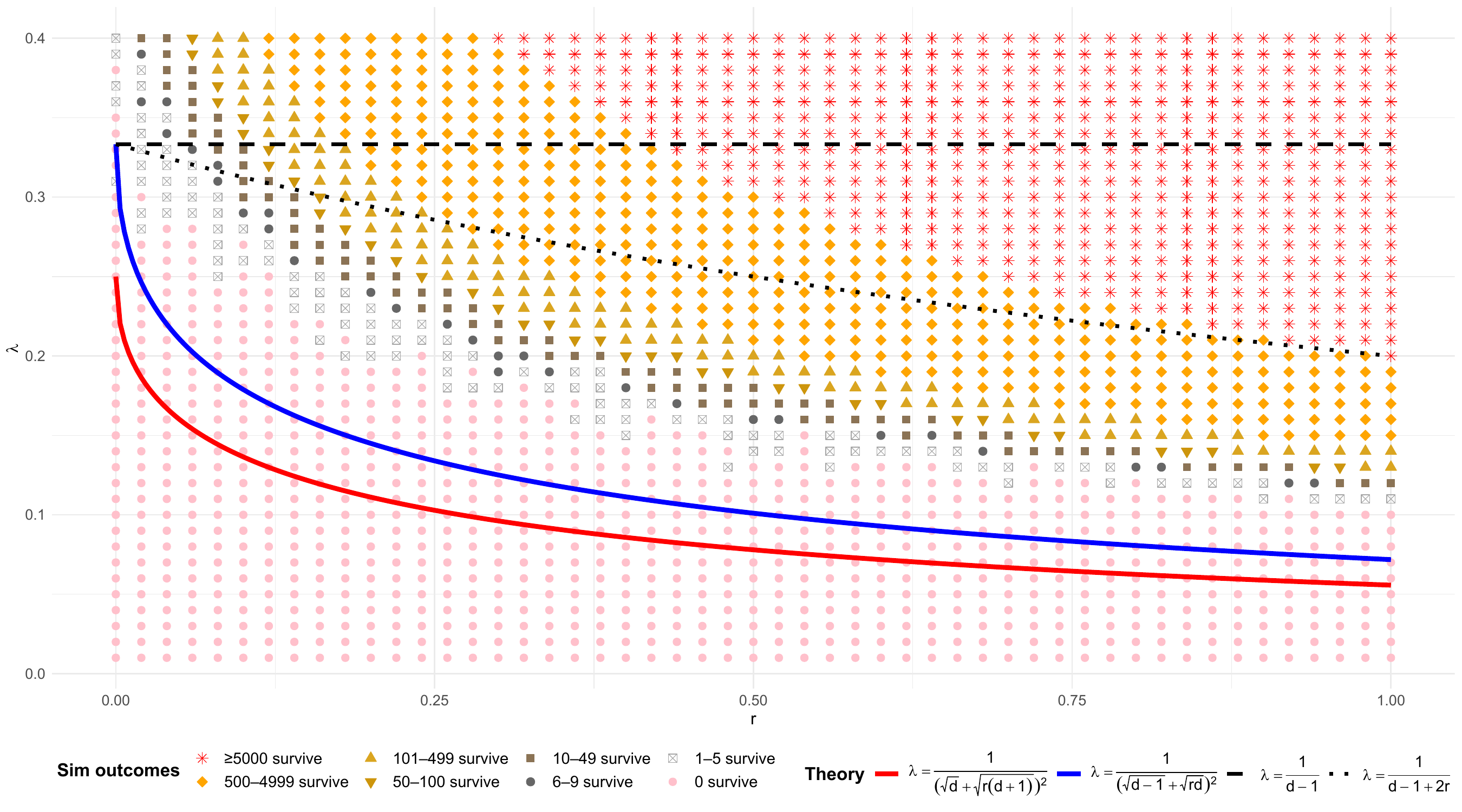}	\caption{ Monte Carlo simulation outcomes for Model $\{\mathbb{T}_4, \lambda, r\}$. 
        Each point corresponds to the empirical result of 30000 independent realizations 
        for a given pair of parameters $(r,\lambda)$, with $r\in[0.01,1]$ and $\lambda\in[0.01,0.4]$. 
        Colors and symbols represent the number of surviving realizations: 
        pink circles denote total extinction (0 surviving runs), while grey to red symbols 
        indicate increasing survival levels (1--5, 6--9, 10--49, 50--100, 101--499, 500--4999, and $\ge 5000$ surviving realizations). The same
        theoretical thresholds for survival from Figure \ref{figura1} are shown for comparison.    
 } 
	\label{figura2}
\end{figure}

\section{Proofs}\label{proofs}

\begin{proof}[Proof of Proposition~\ref{monotonia}] For simplicity, we present the proof for \(d=1\), i.e., for the model \(\{\mathbb{N}_0^+, \lambda, r\}\), where \(\mathbb{N}_0^+\) denotes the directed graph with vertices \(\{0, 1, 2, \dots\}\) and edges \(\{(i, i+1) : i \geq 0\}\), where \((i, i+1)\) indicates a directed edge from vertex \(i\) to \(i+1\).

Let \(0<\lambda_1 < \lambda_2\) and $0<r_1<r_2<1$. We denote by \(\eta_t^1\) and \(\eta_t^2\) the processes \(\{\mathbb{N}_0^+, \lambda_1, r_1\}\) and \(\{\mathbb{N}_0^+, \lambda_2, r_2\}\), respectively. Since in each of these models the creation of a new pathogen is only possible at the neighboring vertex of the furthest vertex from the root of $\mathbb{N}_0^+$ currently occupied by a pathogen, it is possible to jointly construct \(\eta_t^1\) and \(\eta_t^2\) as follows.

Let \(M_i := \sup\{k \geq 0 : \eta_t^i(k) \neq 0\}\) represent the position of the vertex farthest from the root that is occupied by a pathogen at time \(t\) in the process \(\eta_t^i\), for \(i=1,2\). In the process \(\eta^2\), after an exponentially distributed time with rate \(\lambda_2\), the pathogen located at vertex \(M_2\) generates a new pathogen at vertex \(M_2 + 1\). Associated with the same exponential time, with probability \(\lambda_1/\lambda_2\), in the process \(\eta^1\), the pathogen located at \(M_1\) creates a new pathogen at vertex \(M_1 + 1\). 

In the process \(\eta^2\), when a new pathogen is born, an independent random variable \(U \sim \text{UNIF}(0,1)\) is drawn. If \(U > r_2\), the new pathogen is of the same type as the ancestral pathogen; otherwise, the new pathogen is of type \(k+1\), where \(k\) is the type of the ancestral pathogen. If a corresponding birth occurs in the process \(\eta^1\), we use the same variable \(U\), and if \(U > r_1\), the new pathogen in \(\eta^1\) is of the same type as its ancestral pathogen; otherwise, the new pathogen is of type \(l+1\), where \(l\) is the type of the ancestral pathogen.

The death events in both processes, $\eta^1$ and $\eta^2$, follow the same Poisson process with rate 1. At the occurrence times of the Poisson process, all pathogens of the smallest type present in each process, $\eta^1$ and $\eta^2$, are eliminated. 

The previous construction demonstrates that it is possible to couple the processes \(\{\mathbb{N}_0^+, \lambda_1, r_1\}\) and \(\{\mathbb{N}_0^+, \lambda_2, r_2\}\) in such a way that, at all times, the number of pathogens (and pathogen types) in \(\{\mathbb{N}_0^+, \lambda_2, r_2\}\) is always greater than or equal to the number of pathogens (and pathogen types) in \(\{\mathbb{N}_0^+, \lambda_1, r_1\}\).
Thus, the survival probability of \(\{\mathbb{N}_0^+, \lambda_2, r_2\}\) is greater than or equal to the survival probability of \(\{\mathbb{N}_0^+, \lambda_1, r_1\}\).

The arguments presented above can be extended to \(\{\mathbb{T}_d^+, \lambda, r\}\), with \(d \geq 1\). It is enough to observe that the number of pathogens present along each branch (an infinite path starting from the root) of \(\mathbb{T}_d^+\) behaves in the same way as in the process \(\{\mathbb{N}_0^+, \lambda, r\}\).

\end{proof}

The proof below adapts a strategy presented by Aldous and Krebs \cite{aldous} for the survival and extinction of the birth-and-assassination process.

\begin{proof}[Proof of Theorem \ref{T:arbol_dirigido} $(i)$]
Let \(\{ B_i \}_{i \geq 1}\) be independent exponential random variables with rate \(\lambda\), and let \(\{K_i\}_{i \geq 1}\) be a sequence of independent random variables such that \(K_1\) follows an exponential distribution with rate 1. For \(i \geq 2\), the variables \(K_i\) are mixed random variables, being zero with probability \(1 - r\) and following an exponential distribution with rate 1 with probability \(r\).    
Then, the probability that a pathogen is born at a given vertex \(\nu\) at the $k$-th level of $\mathbb{T}_{d}^+$ is equal to
	$$\mathbb{P}\left(\sum_{i=1}^{j} B_i < \sum_{i=1}^j K_i,~j=1,...,k\right).$$ 
To see this, notice that the event
\(
\Bigl\{\sum_{i=1}^{j} B_i < \sum_{i=1}^j K_i,\; j = 1, \dots, k\Bigr\}
\)
ensures that, at each intermediate level \(j\) along the path from the root to the vertex \(\nu\), a pathogen is produced before its progenitor at level \(j-1\) dies.  In particular, the inequality \(B_1 < K_1\) guarantees that the pathogen at the root creates a descendant at level 1 before dying.  The subsequent inequalities, such as \(B_1 + B_2 < K_1 + K_2\), ensure that, cumulatively, the infectious lineage survives up to the second level, and so on, up to level \(k\).

Note that, this probability is the same for every vertex at the $k$-th level and that, there are $d^k$ such vertices. Thus,\\

\noindent
$\mathbb{E}\left(\mbox{total number of pathogens born in } \mathbb{T}_{d}^+\right)=$
\begin{eqnarray*}\label{cheby}
&=&\sum_{v\in \mathbb{T}_d^+}\mathbb{E}[\mathbb{I} \{\mbox{a pathogen is born at}~v\}] \\
&=&\sum_{k\geq 1}d^k\mathbb{P}\left(\sum_{i=1}^j B_i < \sum_{i=1}^j K_i,~j=1,...,k\right)  \\
&\leq &\sum_{k\geq 1}d^k\mathbb{P}\left(\sum_{i=1}^k B_i < \sum_{i=1}^k K_i\right) \\
&\leq &\sum_{k \geq 1}{d^k} \mathbb{E}\left(e^{u(\sum_{i=1}^k K_i-\sum_{i=1}^k B_i)}\right),  \\
&= &\frac{\lambda d}{(\lambda + u)(1-u)}\sum_{k\geq 1}\left[\frac{\lambda d}{\lambda + u}\left(1+\frac{ru}{1-u}\right)\right]^{k-1},
	\end{eqnarray*}

\noindent
where the second inequality is obtained by using Markov's inequality 
for $0<u<1$, and the last expression follows from independence and the properties of the moment generating functions of the random variables $\{K_i \} $ and $\{B_i \}$.\\	Clearly, if  
\[
\inf_{0<u<1} \left\{ \frac{\lambda d}{\lambda + u} \left(1 + \frac{ru}{1-u} \right) \right\} < 1,
\]  
then  
\[
\mathbb{E} \left(\text{total number of pathogens born in } \mathbb{T}_{d}^+ \right) < +\infty.  
\]  
In this case, the process \(\{\mathbb{T}_d^+, \lambda, r\}\) becomes extinct. Finally, the condition  
\[
\inf_{0<u<1} \left\{ \frac{\lambda d}{\lambda + u} \left(1 + \frac{ru}{1-u} \right) \right\} < 1  
\]  
is equivalent to  
\[
\lambda < \left[\frac{\sqrt{d-1} - \sqrt{rd}}{d(1-r) - 1} \right]^2,
\]  
thus establishing the result. 
\end{proof}

To facilitate the proof of item (ii) of Theorem~\ref{T:arbol_dirigido}, it is convenient to define an alternative version of the model \(\{\mathbb{T}_d^+, \lambda, r\}\), in which the initial pathogen, placed at the root of \(\mathbb{T}_d^+\), also has an associated killing time given by a mixed random variable. This variable is 0 with probability \(1 - r\) and follows an exponential distribution with rate 1 with probability \(r\).
We denote this variant of the model by \( \{\mathbb{T}_d^+, \lambda, r\}_{*}.\)
Let $q$ and $q_*$ denote the probability of extinction of \(\{\mathbb{T}_d^+, \lambda, r\}\) and \(\{\mathbb{T}_d^+, \lambda, r\}_*\), respectively. Conditioning on the killing time of the initial pathogen in \(\{\mathbb{T}_d^+, \lambda, r\}_*\) we have that 
 $q_*=(1-r) + rq.$ Thus, 
 $q^*=1$ if and only if $q=1$. Therefore, to prove survival in \(\{\mathbb{T}_d^+, \lambda, r\}\) it is sufficient to show survival in \(\{\mathbb{T}_d^+, \lambda, r\}_*\). Finally, we require the following lemma to study the survival of \(\{\mathbb{T}_d^+, \lambda, r\}_*\). 

\begin{lem}{\cite[Lemma 1]{aldous}}
\label{lemaldous}
Let $X_1, X_2, \ldots$ be i.i.d. random variables with $\mathbb{E}[X] <0$ and $\mathbb{P} [ X>0]>0$. Let $\mathbb{E}[e^{u X}]=\psi(u)$ be finite in some neighborhood of 0, and let $\rho=\inf_{u>0} \psi(u)$. Then,
$$ \lim_{n\rightarrow \infty}\frac{1}{n} \log \mathbb{P} \left[ \sum_{j=1}^k X_j>0, k=1,\ldots, n. \right]= \log \rho. $$ 
\end{lem}

\begin{proof}[Proof of Theorem \ref{T:arbol_dirigido} $(ii)$] 
From the above, it suffices to show that \(\{\mathbb{T}_d^+, \lambda, r\}_*\) survives if \(\lambda > \lambda_c(d, r)\). To this end, we use the fact that \(\lambda > \lambda_c(d, r)\) is equivalent to the condition  
\[
\inf_{0<u<1} \left\{ \frac{\lambda d}{\lambda + u} \left(1 + \frac{ru}{1-u} \right) \right\} > 1.  
\]

Let \(\{ B_i \}_{i \geq 1}\) be independent exponential random variables with rate \(\lambda\), and let \(\{K_i\}_{i \geq 1}\) be a sequence of independent mixed random variables, being zero with probability \(1 - r\) and following an exponential distribution with rate 1 with probability \(r\).  Let $Z_i=K_i-B_i$, $i=1,\ldots$. Thus, in the process $\{\mathbb{T}_d^+,\lambda,r\}_*$, the probability that a pathogen is born at a fixed vertex at the $k$-th is given by  
\begin{eqnarray*}\label{eq:Probverticefixo}	
		\mathbb{P}\left(\sum_{i=1}^j Z_i>0,~j=1,\ldots,k\right).
\end{eqnarray*}
 
Observe that, $\mathbb{E}[Z]=r-\frac{1}{\lambda}<0$ if $r \lambda <1$. So, by Lemma \ref{lemaldous},	

\[ \lim_{k\rightarrow\infty}\frac{1}{k}\log\mathbb{P}\left(\sum_{i=1}^j Z_i>0,~j=1,\ldots,k\right)= \hfill \] \[\hfill = \log\left[\inf_{0<u<1}\left\{\frac{\lambda }{\lambda + u}\left(1+\frac{ru}{1-u}\right)\right\}\right].\]

Moreover, by assumption, for some $\delta>0$,
$$\inf_{0<u<1}\left\{\frac{d\lambda }{\lambda + u}\left(1+\frac{ru}{1-u}\right)\right\}=1+\delta.$$
Suppose for now that $r \lambda <1$ and take $\epsilon=\delta/2$. Then, there exists $K\in\mathbb{N}$ such that for all $k\geq K$,

\begin{eqnarray}\label{probpatogk}
\mathbb{P}\left(\sum_{i=1}^j Z_i>0,~j=1,\ldots, k\right)
&>&\left[\inf_{0<u<1}\left\{\frac{\lambda }{\lambda + u}\left(1+\frac{ru}{1-u}\right)\right\}-\frac{\epsilon}{d}\right]^k \nonumber\\
&=&\frac{1}{d^k}\left[\inf_{0<u<1}\left\{\frac{d\lambda }{\lambda + u}\left(1+\frac{ru}{1-u}\right)\right\}-\epsilon\right]^k \nonumber\\
&=&\frac{1}{d^k}\left(1+\frac{\delta}{2}\right)^{k}.
\end{eqnarray}

 Let $k\in \mathbb{N}$, we say that  a pathogen in $\{\mathbb{T}_d^+,\lambda,r\}_*$ is an \textit{special particle} if it is the initial pathogen placed in the root of $\mathbb{T}_d^+$, or if it is a new mutation that arises at the \(nk\)-th level of \(\mathbb{T}_d^+\), descending from a special particle at the \((n-1)k\)-th level, with no living ancestors.		Let $W_n$ denote the number of special particles in the generation $nk$ of $\{\mathbb{T}_d^+,\lambda,r\}_*$. 
The process $(W_n)_{n\geq0}$ defines a Galton-Watson branching process with $W_0=1$. For each vertex $\nu\in\mathbb{T}_d^+$, the probability that a pathogen is born at vertex $\nu$ and is a mutation is $r\mathbb{P}(\mbox{a pathogen is born at}~v).$    
Thus, the expected number of special particles in generation $k$ in $\{\mathbb{T}_d^+,\lambda,r\}_*$ is given by $$\mathbb{E}[W_1]=d^kr\mathbb{P}\left(\sum_{i=1}^j Z_i>0,~j=1,\ldots, k\right).$$
By (\ref{probpatogk}), we see that the last quantity is greater than 1 for some $k$ sufficiently large. So, the process $(W_n)_{n\geq0}$ is a supercritical branching process. As a consequence $\{\mathbb{T}_d^+,\lambda,r\}_*$ survives.
To conclude, Proposition~\ref{monotonia} shows  that the survival probability is non-decreasing in $\lambda$ and $r$. Therefore, this result also holds for $r \lambda \geq 1$, provided that $$\inf_{0<u<1}\left\{\frac{d\lambda}{\lambda + u}\left(1+\frac{ru}{1-u}\right)\right\} > 1.$$
\end{proof}

\begin{proof}[Proof of Corollary \ref{C:arbol}]
First we define a coupling between the processes 
$\{\mathbb{T}_d,\lambda,r\}$ and $\{\mathbb{T}_d^+,\lambda,r\}$ in such a way that the latter is stochastically dominated by the former. Every pathogen in $\{\mathbb{T}_d^+,\lambda,r\}$ is associated to a pathogen in $\{\mathbb{T}_d,\lambda,r\}$. In the model $\{\mathbb{T}_d, \lambda, r\}$, whenever a pathogen at vertex \( x \) attempts to create a new pathogen at a neighboring vertex that is closer to the root than \( x \), it will succeed in creating the new pathogen, provided that the target vertex is empty. In contrast, in the model $\{\mathbb{T}_d^+, \lambda, r\}$, such birth attempts are not possible. On the other hand, all births that occur in $\{\mathbb{T}_d^+, \lambda, r\}$, occur in $\{\mathbb{T}_d, \lambda, r\}.$ As a consequence, if the process $\{\mathbb{T}_d^+,\lambda,r\}$ survives, the same happens to $\{\mathbb{T}_d,\lambda,r\}$.

Next, we define a coupling between the processes $\{\mathbb{T}_d, \lambda, r\}$ and $\{\mathbb{T}_{d+1}^+, \lambda, r\}$ such that the former is stochastically dominated by the latter. Each pathogen in $\{\mathbb{T}_d, \lambda, r\}$ can be associated with a pathogen in $\{\mathbb{T}_{d+1}^+, \lambda, r\}$. In the model $\{\mathbb{T}_d, \lambda, r\}$, we associate the neighboring vertex to \( x \) that is closer to the root with the extra vertex in the model $\{\mathbb{T}_{d+1}^+, \lambda, r\}$. Thus, in the model $\{\mathbb{T}_d, \lambda, r\}$, whenever a pathogen at vertex \( x \) attempts to create a new pathogen at a neighboring vertex that is closer to the root than \( x \), it will succeed in creating the new pathogen, provided that the target vertex is empty. In contrast, in the model $\{\mathbb{T}_{d+1}^+, \lambda, r\}$, such a birth attempt occurs at the extra vertex. As a consequence if the process $\{\mathbb{T}_{d+1}^+,\lambda,r\}$ dies out, the same happens to $\{\mathbb{T}_{d},\lambda,r\}$.

Finally, the result follows from Theorem~\ref{T:arbol_dirigido} and the fact that the process $\{\mathbb{T}_d^+, \lambda, r\}$ is stochastically dominated by the process $\{\mathbb{T}_d, \lambda, r\}$, which in turn is stochastically dominated by the process $\{\mathbb{T}_{d+1}^+, \lambda, r\}$.
\end{proof}

\section*{Acknowledgements}  
The authors thank the editor and the referee for their insightful and valuable comments, which have greatly improved the manuscript. This work was supported by the ANID–FONDECYT Iniciación grant No. 11230220 and by Universidad de Antioquia (project No. 2025-80410).

\end{document}